\documentclass[11pt]{article}

\usepackage[a4paper,margin=2.8cm]{geometry}
\usepackage[T1]{fontenc}
\usepackage[utf8]{inputenc}
\usepackage[english]{babel}

\usepackage{lmodern}
\usepackage{microtype}
\usepackage{enumitem}
\usepackage{amsmath,amssymb,amsthm,mathtools}
\usepackage{bm}
\usepackage{hyperref}
\hypersetup{colorlinks=true,linkcolor=blue,citecolor=blue,urlcolor=blue}
\usepackage{enumitem}
\usepackage{authblk}

\theoremstyle{plain}
\newtheorem{theorem}{Theorem}[section]
\newtheorem{proposition}[theorem]{Proposition}
\newtheorem{lemma}[theorem]{Lemma}
\newtheorem{corollary}[theorem]{Corollary}
\newtheorem*{theorem*}{Theorem}

\theoremstyle{definition}

\theoremstyle{remark}
\newtheorem{remark}[theorem]{Remark}

\DeclareMathOperator{\im}{im}

\DeclareMathOperator{\dom}{dom}
\newcommand{\CC}{\mathbb{C}}
\newcommand{\ZZ}{\mathbb{Z}}  %

\newcommand{\Z}{\ZZ} 

\newcommand{\R}{\mathbb{R}}

\title{\LARGE\bf Choreographies with Dihedral Symmetry in the Planar 
$n$-Body Problem}

\author[1]{Juan Manuel Sánchez-Cerritos\thanks{Corresponding author: \texttt{jmsc@xanum.uam.mx}}}
\affil[1]{Departamento de Matemáticas, Universidad Autónoma Metropolitana, Ciudad de México, 09310, México}

\date{} 
\begin{document}
\maketitle
\begin{abstract}
We prove the existence of planar $D_n$--equivariant choreographies in the $n$--body problem with homogeneous potential of degree $-\alpha$, $0<\alpha<2$. 
Each body follows the same closed path, rotated and time--shifted, forming a  choreography whenever the winding number $W$ is coprime with $n$. 
Using Mawhin's coincidence degree, we establish collision--free periodic solutions under a simple nonresonance condition. 
The proof relies on the spectral structure of the linearized operator, symmetry--induced separation of the bodies, and uniform energy bounds ensuring compactness of the nonlinear term. 
This provides a topological route to choreographies beyond variational and numerical frameworks.
\end{abstract}



\section{Introduction}
\label{sec:intro}

The planar $n$--body problem in Celestial Mechanics with a homogeneous potential of degree $-\alpha$, $0<\alpha<2$, is a classical field where dynamical analysis, bifurcation theory, and variational methods meet. Among its remarkable configurations are the \emph{choreographies}, introduced by Chenciner and Montgomery~\cite{ChencinerMontgomery2000}, in which all bodies move along the same closed curve in the plane, uniformly shifted in time. The discovery of the celebrated “figure--eight’’ choreography for three equal masses opened an active research line combining symmetry principles, action--minimization techniques, and topological arguments.

Subsequent studies have revealed numerous families of choreographies: Simó~\cite{Simo2001} explored a wide range of them through numerical computations; Ferrario and Terracini~\cite{FerrarioTerracini2004} applied variational principles with symmetry constraints; Kapela and Zgliczyński~\cite{KapelaZgliczynski2003} provided computer--assisted rigorous proofs of existence; Chen and Ouyang~\cite{ChenOuyang2004} investigated bifurcations from central configurations, among others. Many of these solutions exhibit polygonal or dihedral symmetries, suggesting that choreographies with $D_n$ symmetry may have implicitly emerged in earlier studies.

In this paper we adopt a complementary perspective: we establish the existence of $D_n$--equivariant choreographies by means of the \emph{Mawhin coincidence degree theorem}~\cite{Mawhin1979}. This fundamental result in nonlinear analysis applies to operator equations of the form
\[
L u = N(u),
\]
where $L$ is a linear Fredholm operator of index zero and $N$ is an $L$--compact nonlinear perturbation. The theorem guarantees the existence of solutions under suitable a~priori bounds, providing a topological alternative to variational and perturbative methods. Although this framework has been applied to various nonlinear differential equations, its use in Celestial Mechanics and in the context of choreographies remains relatively unexplored. Our goal is to fill this methodological gap and to show that Mawhin’s degree provides a unified scheme to handle symmetry, compactness, and collision exclusion simultaneously.

The analysis relies on two geometric ingredients:
\begin{enumerate}
\item \textbf{Polar description of $D_n$ symmetry.}  
Each trajectory can be represented in polar form $u(t)=r(t)(\cos t,\sin t)$, 
where the radial function $r(t)$ encodes the dihedral symmetry pattern. 
In particular, $D_n$--equivariance corresponds to a periodic and reflection--symmetric 
behavior of $r(t)$, a property that later ensures a uniform separation between the bodies.
\item \textbf{Winding number around the origin.}  
Beyond the basic case where the curve winds once around the origin ($W=1$),
we also consider higher winding numbers $W>1$, which generate more intricate families of choreographies.
\end{enumerate}

Mawhin’s coincidence degree has proved to be a robust tool for establishing periodic solutions of nonlinear differential equations~\cite{Mawhin1972,Mawhin1977,Mawhin1979,CapiettoMawhinZanolin1992,Ramos2020,Santos2021,Liu2023}, including problems with gyroscopic or resonance effects analogous to those appearing in the $n$--body equations.

The main contribution of this paper is to prove, under a natural nonresonance condition, the existence of $D_n$--equivariant choreographies for any winding number $W\ge1$ such that $W$ and $n$ are coprime. This arithmetic restriction ensures that the configuration corresponds to a single shared orbit rather than multiple identical sub--orbits. Formally, we establish the following result.

\begin{theorem*}[Existence of $D_n$--equivariant choreographies]
Let $n\ge3$ and $W\ge1$ be coprime integers ($\gcd(W,n)=1$), and consider the planar $n$--body problem with homogeneous potential $r^{-\alpha}$, $0<\alpha<2$. 
Assume the nonresonance condition $\Omega \neq \pm k\,\tfrac{2\pi}{T}$ for all $k\in\mathbb N$. 
Then there exists at least one $T$--periodic function $u\in H^2_{\mathrm{per}}([0,T];\mathbb R^2)$ such that
\[
L u = N(u),
\]
where $L:H^2_{\mathrm{per}}\to L^2_{\mathrm{per}}$ is the linear Fredholm operator associated with the rotating frame. 
The corresponding configuration
\[
q_k(t)=R_{2\pi k/n}\,u\!\left(t+\tfrac{kT}{n}\right),\qquad k=0,\dots,n-1,
\]
defines a collision--free $D_n$--equivariant choreography with winding number $W$.
\end{theorem*}

\medskip

Section~\ref{sec:symmetry} introduces the $D_n$ symmetry, the notion of winding number, and the reduction of the $n$--body system to a single functional equation for the generating curve. 
Section~\ref{sec:functional} formulates the problem in a rigorous functional setting, defining the periodic Sobolev (Hilbert) spaces and the linear and nonlinear operators that lead to the abstract equation $L u = N(u)$. 
Section~\ref{sec:mawhinframework} presents Mawhin’s coincidence framework, providing the topological degree setting in which the existence argument is carried out. 
Section~\ref{sec:operatorL} develops the spectral analysis of the linear operator $L:H^2_{\mathrm{per}}\to L^2_{\mathrm{per}}$, establishing its Fredholm and coercive properties under the nonresonance condition. 
In Section~\ref{sec:apriori}, we derive uniform energy estimates for the homotopy $L u = \lambda N(u)$ that guarantee boundedness and collision exclusion. 
Section~\ref{sec:operatorN} analyses the nonlinear operator $N$, proving its continuity, compactness, and local Lipschitz bounds on bounded sets. 
Finally, the concluding Section~\ref{sec:mawhin} applies Mawhin’s coincidence degree to combine all these ingredients and establish the existence of $D_n$--equivariant, collision--free choreographies for any coprime pair $(W,n)$.

\section{Dihedral symmetry $D_n$ and problem formulation}
\label{sec:symmetry}

\subsection{$D_n$–equivariant choreographies and the winding number $W$}

We consider the planar $n$--body problem with equal masses $m>0$, whose positions
$q_i:\mathbb{R}\to\mathbb{R}^2$ satisfy Newton’s equations
\begin{equation}\label{eq:newton}
m\ddot q_i(t)
= \sum_{\substack{j=0 \\ j\neq i}}^{n-1}
\frac{m^2\,(q_j(t)-q_i(t))}{\|q_j(t)-q_i(t)\|^{\alpha+2}},
\qquad i=0,\ldots,n-1,
\end{equation}
where $0<\alpha<2$ is the exponent of the homogeneous potential $V(r)=r^{-\alpha}$.
We seek periodic solutions of~\eqref{eq:newton} exhibiting the \emph{choreographic symmetry}, namely that all bodies follow the same closed trajectory $u(t)$ with a fixed time shift.

Let $R_{2\pi/n}$ denote the rotation in the plane by angle $2\pi/n$, 
and fix the reflection 
\[
S=\begin{pmatrix}1&0\\[2pt]0&-1\end{pmatrix},
\]
so that $SJS=-J$, where $J=\begin{pmatrix}0&-1\\1&0\end{pmatrix}$.
Let $T>0$ denote the period.

We define a \emph{choreographic configuration} by
\begin{equation}\label{eq:choreo}
q_i(t) = R_{2\pi i/n}\,u\!\left(t+\frac{iT}{n}\right),
\qquad i=0,\ldots,n-1,
\end{equation}
where $u:\mathbb{R}\to\mathbb{R}^2$ is a sufficiently regular $T$--periodic curve.
This condition imposes a discrete symmetry isomorphic to the dihedral group $D_n$,
generated by the rotation $R_{2\pi/n}$ and the space--time involution 
$\mathcal S:u(t)\mapsto S\,u(-t)$.
The set of such curves forms the closed subspace
\[
X_{D_n} =
\Bigl\{
u\in H^2_{\mathrm{per}}([0,T];\mathbb{R}^2) :
u\!\bigl(t+\tfrac{T}{n}\bigr) = R_{2\pi/n}u(t),\;
u(-t)=S\,u(t)
\Bigr\}.
\]
The embedding $H^2_{\mathrm{per}}\hookrightarrow C^1_{\mathrm{per}}$ guarantees that
the pointwise evaluations in~\eqref{eq:choreo} are well defined.  
In this space, a single function $u$ completely determines the motion of the $n$ bodies.

\medskip
The \emph{winding number} $W\in\mathbb{N}$ records how many full turns the curve $u(t)$
performs around the origin during one period $T$.
Writing $u(t)=r(t)e^{i\theta(t)}$, we set
\[
W = \frac{\theta(T)-\theta(0)}{2\pi} \in\mathbb{Z}.
\]
The parameter $W$ controls the topology of the generating curve and will be fixed throughout the analysis.
When $W\ge1$, the trajectories may have self--intersections and a
“wrapped polygon’’ structure, analogous to limacón--type choreographies.

\medskip
The arithmetic relation between $W$ and $n$ determines the nature of the choreography.
Indeed, the $D_n$ symmetry imposes
\[
u\!\left(t+\tfrac{T}{n}\right)=R_{2\pi/n}\,u(t),
\]
so that after $n$ iterations one winds $W$ times around the origin.
If $W$ and $n$ are coprime, the action generated by
\[
(t,u)\mapsto \bigl(t+\tfrac{T}{n},\,R_{2\pi/n}u\bigr)
\]
is transitive on the $n$ bodies, hence they all traverse the same planar curve.
Conversely, when $W$ and $n$ have a common divisor $d>1$, the motion splits into $d$ disjoint subsets of bodies,
each subset following a distinct curve congruent under rotations.

\begin{proposition}[Coprimality condition for choreographies]
\label{prop:gcd}
Let $n\ge3$ and $W\ge1$ be integers, and consider the $D_n$--equivariant configuration given by~\eqref{eq:choreo}. Here $\gcd(W,n)$ denotes the greatest common divisor of $W$
 and $n$. Then:
\begin{enumerate}[label=(\roman*)]
\item If $\gcd(W,n)=1$, the configuration is a \emph{choreography}: all bodies trace the same curve $u(t)$, time--shifted by $T/n$.
\item If $\gcd(W,n)=d>1$, the configuration decomposes into $d$ \emph{disjoint sub--choreographies}, each formed by $n/d$ bodies sharing a common trajectory.
\end{enumerate}
\end{proposition}

\begin{proof}
Iterating the symmetry $u(t+\tfrac{T}{n})=R_{2\pi/n}u(t)$ yields
\[
u\!\left(t+\tfrac{jT}{n}\right)=R_{2\pi j/n}u(t),\qquad j\in\mathbb{Z}.
\]
The time shift $t\mapsto t+\tfrac{T}{n}$ combined with $R_{2\pi/n}$ generates an action of the cyclic group $\mathbb{Z}_n$ on the set of bodies $\{q_0,\dots,q_{n-1}\}$. 
The number of orbits of this action equals $\gcd(W,n)$.  
If $\gcd(W,n)=1$, the generator acts transitively, hence all bodies belong to a single orbit, i.e., a single generating curve.  
If $\gcd(W,n)=d>1$, the action splits into $d$ disjoint orbits of size $n/d$, each corresponding to a distinct curve traced by $n/d$ bodies. 
\end{proof}

The condition $\gcd(W,n)=1$ ensures that there are no exact phase repetitions between different bodies,
avoiding collisions due to temporal overlaps.
When $\gcd(W,n)>1$, spatial coincidences may occur between bodies in different sub--choreographies,
although each subconfiguration remains collision--free internally.

\medskip
\medskip
In what follows we restrict attention to the \emph{coprime} case $\gcd(W,n)=1$, 
so that the configuration~\eqref{eq:choreo} 
indeed represents a choreography in which all bodies share a single generating curve $u(t)$, 
uniformly shifted in time. 
Under this hypothesis, all subsequent expressions and results are to be understood 
in this transitively symmetric regime.

\begin{remark}[Spectral structure of the $D_n$–symmetry]
\label{rem:spectral-XDn}
The condition 
\[
u\!\left(t+\tfrac{T}{n}\right)=R_{2\pi/n}u(t)
\]
imposes a specific constraint on the Fourier spectrum of $u$.
Writing $u(t)$ in complex form as
\[
u(t)=\sum_{k\in\mathbb{Z}} a_k\,e^{i k\omega t},
\qquad \omega=\tfrac{2\pi}{T},
\]
we obtain
\[
e^{i k\omega T/n}\,a_k = e^{i 2\pi/n}\,a_k,
\]
so that only the modes satisfying $k\equiv 1\pmod n$ may have nonzero coefficients.
Consequently, admissible functions in $X_{D_n}$ contain only
Fourier modes of the form $k=1+\ell n$ with $\ell\in\mathbb{Z}$.
The dominant mode $k=1+\ell n$ determines the winding number 
$W=1+\ell n$, which is automatically coprime with $n$.
This characterization links the dihedral symmetry of the configuration space
to its harmonic structure and explains the multi–winding
patterns observed in the generating curves illustrated later.
\end{remark}

\medskip
Representative examples of generating curves $u(t)$ satisfying
the dihedral symmetry constraint 
$u(t+\tfrac{T}{n})=R_{2\pi/n}u(t)$
and having winding number $W>1$ are shown in
Figure~\ref{fig:generating-curves}.
Each curve corresponds to a dominant Fourier mode
$k=1+\ell n$, consistent with the spectral characterization
in Remark~\ref{rem:spectral-XDn}.
They illustrate the multi–winding, $D_n$–equivariant
patterns that define the configuration space $X_{D_n}$.

\begin{figure}[t]
  \centering
  \includegraphics[width=0.9\textwidth]{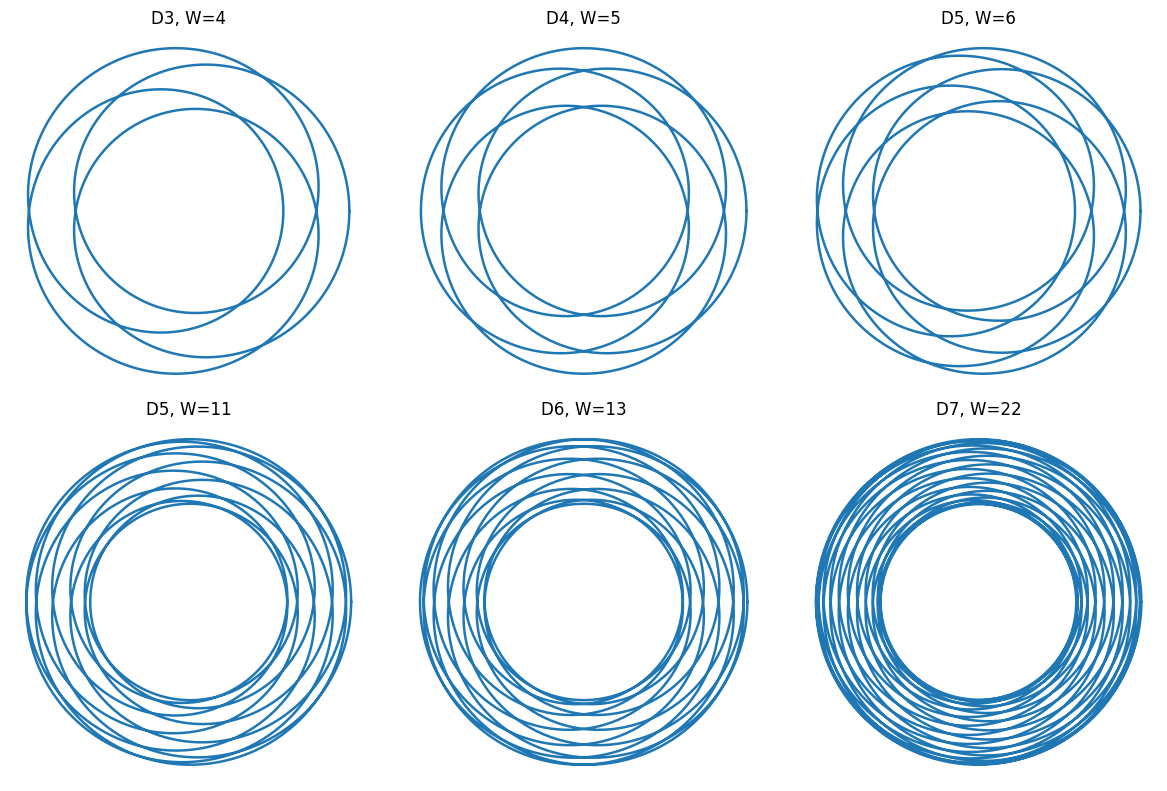}
  \caption{
  Generating curves $u(t)$ with dihedral symmetry 
  $u(t+\tfrac{T}{n})=R_{2\pi/n}u(t)$ 
  and winding number $W>1$.
  Each curve is composed of admissible Fourier modes
  $k=1+\ell n$, ensuring that $\gcd(W,n)=1$.
  These examples correspond to $(n,W)=(3,4)$, $(4,5)$, $(5,6)$,
  $(5,11)$, $(6,13)$, and $(7,22)$, respectively.}
  \label{fig:generating-curves}
\end{figure}

\subsection{Equations of motion in the rotating frame}

To study periodic solutions with $D_n$ symmetry, it is convenient to switch to a \emph{rotating frame} with constant angular velocity $\Omega\in\mathbb{R}$.
Writing $q_i(t)=R_{\Omega t}\,x_i(t)$, Newton’s equations~\eqref{eq:newton} become
\begin{equation}\label{eq:rotating}
\ddot x_i + 2\Omega J \dot x_i - \Omega^2 x_i
= \sum_{\substack{j=0 \\ j\neq i}}^{n-1}
\frac{m\,(x_j - x_i)}{\|x_j - x_i\|^{\alpha+2}},
\qquad i=0,\ldots,n-1.
\end{equation}
Here, the gyroscopic term $2\Omega J\dot x_i$ and the centrifugal term $-\Omega^2 x_i$
represent the fictitious forces in the rotating frame.

\medskip

Under the choreographic condition~\eqref{eq:choreo}, the system~\eqref{eq:rotating}
reduces to a single equation for the generating curve \(u(t)\):
\begin{equation}\label{eq:equation-u-physical}
\ddot u + 2\Omega J\dot u - \Omega^2 u
= \sum_{k=1}^{n-1}
\frac{m\,(R_{2\pi k/n}u(t+\tfrac{kT}{n}) - u(t))}{
\|R_{2\pi k/n}u(t+\tfrac{kT}{n}) - u(t)\|^{\alpha+2}}.
\end{equation}

To express the problem in the Fredholm form required by Mawhin’s theorem,
we multiply both sides of~\eqref{eq:equation-u-physical}, by $-1$, thus obtaining
\begin{equation}\label{eq:equation-u}
-\,\ddot u - 2\Omega J\dot u + \Omega^2 u
= - \sum_{k=1}^{n-1}
\frac{m\,(R_{2\pi k/n}u(t+\tfrac{kT}{n}) - u(t))}{
\|R_{2\pi k/n}u(t+\tfrac{kT}{n}) - u(t)\|^{\alpha+2}}.
\end{equation}

\medskip

We define the linear operator
\begin{equation}\label{eq:L-def}
L u = -\ddot u - 2\Omega J\dot u + \Omega^2 u,
\qquad
L:H^2_{\mathrm{per}}([0,T];\mathbb{R}^2)\longrightarrow L^2_{\mathrm{per}}([0,T];\mathbb{R}^2),
\end{equation}
and the nonlinear operator $N:X_{D_n}\to L^2_{\mathrm{per}}$ given by the right--hand side of~\eqref{eq:equation-u}, which couples the $n-1$ rotated and time--shifted copies of \(u\).

Thus, the $D_n$--choreography problem in the rotating frame takes the compact operator form
\begin{equation}\label{eq:LuNu}
L u = N(u), \qquad u\in X_{D_n}\subset H^2_{\mathrm{per}}([0,T];\mathbb{R}^2),
\end{equation}
which provides the functional framework for the application of Mawhin’s coincidence degree.

\subsection{Uniform separation and collision exclusion}

A key point for applying Mawhin’s theorem is to ensure that the choreographic trajectories under consideration remain collision--free.
Geometrically, the $D_n$ symmetry imposes a minimal separation between bodies, independent of the winding number $W\ge1$. 
This is formalized as follows.

\begin{proposition}[Uniform separation under $D_n$ symmetry for all $W\ge1$]
\label{prop:uniform-sep-W}
Let $n\ge3$ and consider $\gamma:\mathbb{R}\to\mathbb{R}^2$ of the form
\[
\gamma(s)=\rho(s)\,(\cos s,\sin s),
\qquad
\rho\in C^1(\mathbb{R}),\quad \rho(s+2\pi/n)=\rho(s),\quad \rho(-s)=\rho(s),\quad \rho(s)\ge r_*>0.
\]
For a winding number $W\ge1$, define the choreographic trajectories
\[
q_k(t)=R_{2\pi k/n}\,\gamma\!\Big(Wt+\tfrac{2\pi k}{n}\Big),\qquad k=0,\dots,n-1.
\]
Then, for all $t\in\mathbb{R}$ and all $k=1,\dots,n-1$,
\begin{equation}\label{eq:rho0-sep}
\|q_k(t)-q_0(t)\|\ \ge\ 2\,r_*\,\sin\!\Big(\tfrac{\pi}{n}\Big)\ =:\ \rho_0>0.
\end{equation}
In particular, the configuration is collision--free and uniformly separated,
with a bound that does not depend on $W$.
\end{proposition}

\begin{proof}
By definition,
\[
q_k(t)-q_0(t)=R_{2\pi k/n}\,\gamma\!\Big(Wt+\tfrac{2\pi k}{n}\Big)-\gamma(Wt).
\]
Using the $2\pi/n$--periodicity of $\rho$ and the identity 
$R_{2\pi k/n}(\cos s,\sin s)=(\cos(s+\tfrac{2\pi k}{n}),\sin(s+\tfrac{2\pi k}{n}))$,
we obtain
\[
q_k(t)-q_0(t)=(R_{2\pi k/n}-I)\,\gamma(Wt).
\]
For any rotation $R_\phi$ and vector $x\in\mathbb{R}^2$ one has
$\|(I-R_\phi)x\|=2\|x\|\sin(\phi/2)$, hence
\[
\|q_k(t)-q_0(t)\|=2\,\|\gamma(Wt)\|\sin\!\Big(\tfrac{\pi k}{n}\Big).
\]
Since $\|\gamma(Wt)\|=\rho(Wt)\ge r_*$ and
$\min_{1\le k\le n-1}\sin(\pi k/n)=\sin(\pi/n)$, it follows that
\[
\|q_k(t)-q_0(t)\|\ \ge\ 2\,r_*\,\sin\!\Big(\tfrac{\pi}{n}\Big)=:\rho_0>0.
\]
\end{proof}

The estimate~\eqref{eq:rho0-sep} depends only on $n$ and the minimal radius $r_*$,
but not on the winding number $W$. 
Thus, even for choreographies performing multiple turns around the origin,
the minimal pairwise separation remains uniform.

\begin{corollary}[Collision exclusion]
Under the hypotheses of Proposition~\ref{prop:uniform-sep-W},
all trajectories $q_k$ are collision--free for all times.
Consequently, the nonlinear operator $N(u)$ is well defined on $X_{D_n}$,
since the denominators 
$\|u(t)-R_{2\pi k/n}\,u(t+\tfrac{kT}{n})\|^{\alpha+2}$
are uniformly bounded away from zero.
\end{corollary}

\section{Functional framework and operators}
\label{sec:functional}

The goal of this section is to formulate the $D_n$--equivariant choreography problem
in a rigorous functional setting, so that the equation
\[
L u = N(u)
\]
is well posed as an identity between Hilbert (hence Banach) spaces.
We specify the functional spaces $X$ and $Z$,
define the linear and nonlinear operators $L$ and $N$,
and express $L$ in a compact form using the \emph{rotating covariant derivative}
$D_t=\partial_t+\Omega J$, which reveals its geometric and coercive structure.

\subsection{Spaces $X$ and $Z$}

Let $T>0$ be the period of the generating curve $u(t)$.
We work with Sobolev spaces of $T$--periodic functions taking values in $\mathbb{R}^2$.

We define
\[
X = H^2_{\mathrm{per}}([0,T];\mathbb{R}^2)
= \{u\in H^2([0,T];\mathbb{R}^2): u(0)=u(T),\, \dot u(0)=\dot u(T)\},
\]
endowed with the norm
\[
\|u\|_X^2 = \|u\|_{H^2}^2
= \|u\|_{L^2}^2 + \|\dot u\|_{L^2}^2 + \|\ddot u\|_{L^2}^2.
\]
We additionally impose the \emph{zero--mean} condition
\[
\int_0^T u(t)\,dt = 0,
\]
which removes the translational invariance and guarantees the injectivity of $L$ under nonresonance.

The target space is
\[
Z = L^2_{\mathrm{per}}([0,T];\mathbb{R}^2),
\qquad \|v\|_Z = \|v\|_{L^2}.
\]
The pair $(X,Z)$ forms a Hilbert couple, and the embedding
$X\hookrightarrow Z$ is compact by the Rellich--Kondrachov theorem \cite{Brezis2010}.
Thus, the embedding $X\hookrightarrow Z$ being compact implies that 
$L^{-1}:Z\to Z$ is compact.

\medskip
The $D_n$--equivariant symmetry is enforced by restricting $X$ to the closed subspace
\[
X_{D_n} =
\Bigl\{
u\in H^2_{\mathrm{per}}([0,T];\mathbb{R}^2):
u\!\bigl(t+\tfrac{T}{n}\bigr)=R_{2\pi/n}u(t),\;
u(-t)=S\,u(t)
\Bigr\}.
\]
This subspace is invariant under both $L$ and $N$ because
$SJS=-J$ and the interaction terms are rotation--equivariant.
The existence problem will be analyzed entirely within $X_{D_n}$.

\subsection{Definition of the operators $L$ and $N$}

On $X_{D_n}$consider the \emph{linear operator}
\begin{equation}\label{eq:L-def}
L u = -\ddot u - 2\Omega J \dot u + \Omega^2 u,
\qquad
L:X\to Z.
\end{equation}
The minus sign is introduced for analytical convenience:
it renders $L$ \emph{coercive} and interpretable as the negative of the
dynamical operator $\partial_t^2 + 2\Omega J\partial_t - \Omega^2$.
The principal part of $L$ is the elliptic operator $-\partial_t^2$,
so $L$ is a Fredholm operator of index zero from $H^2_{\mathrm{per}}$ to $L^2_{\mathrm{per}}$.

\medskip
The \emph{nonlinear operator} $N:X\to Z$ is given by
\begin{equation}\label{eq:N-def}
(Nu)(t)
= \sum_{k=1}^{n-1}
\frac{m\,\bigl(u(t)-R_{2\pi k/n}\,u(t+\tfrac{kT}{n})\bigr)}{
\|u(t)-R_{2\pi k/n}\,u(t+\tfrac{kT}{n})\|^{\alpha+2}},
\qquad 0<\alpha<2.
\end{equation}
Each term represents the gravitational interaction between the reference particle
and its rotated copy $R_{2\pi k/n}u(t+\tfrac{kT}{n})$.
By Proposition~\ref{prop:uniform-sep-W},
the denominators remain uniformly bounded away from zero,
so $N$ is well defined and continuous on $X_{D_n}$.
Moreover, since rotations and time shifts commute with differentiation,
$N$ preserves the $D_n$ symmetry:
if $u\in X_{D_n}$ then $N(u)\in X_{D_n}$.

\medskip
Thus, the $D_n$ choreography problem takes the abstract form
\begin{equation}\label{eq:Lu=Nu-abstract}
L u = N(u), \qquad u\in X_{D_n}\subset H^2_{\mathrm{per}}([0,T];\mathbb{R}^2),
\end{equation}
which is the functional formulation required for the application of
Mawhin’s coincidence theorem.

\subsection{Covariant derivative and compact form of $L$}

The operator $L$ can be written compactly
in terms of the \emph{rotating covariant derivative}
\begin{equation}\label{eq:Dt-def}
D_t = \partial_t + \Omega J.
\end{equation}
A direct computation gives
\[
D_t^2
= \partial_t^2 + 2\Omega J\partial_t + \Omega^2 J^2
= \partial_t^2 + 2\Omega J\partial_t - \Omega^2 I,
\]
so that
\begin{equation}\label{eq:L-Dt}
L = -D_t^2.
\end{equation}
Hence $L$ is the \emph{negative covariant Laplacian} associated with
the flat connection determined by the constant matrix $\Omega J$.
In terms of $D_t$, the abstract equation~\eqref{eq:Lu=Nu-abstract} reads
\[
-D_t^2 u = N(u),
\]
so that $D_t^2u$ represents the covariant acceleration
as measured in the rotating frame.

\subsection{Geometric and energetic interpretation}

The operator $D_t$ has a natural geometric meaning as a covariant derivative
associated with the rotating connection on the trivial bundle
$[0,T]\times\mathbb{R}^2\to[0,T]$,
whose fibers rotate with angular velocity $\Omega$.

In the inertial frame, $\dot u$ measures the absolute variation of $u(t)$,
while in the rotating frame an additional change arises from the rotation of the basis.
This correction is expressed by the term $\Omega Ju$,
so that the effective derivative in the rotating frame is
\[
D_t u = \dot u + \Omega J u.
\]
If $D_tu=0$, the vector $u(t)$ co--rotates with the frame,
maintaining its orientation relative to it.

\medskip
The associated energy identity is obtained by integration by parts:
\[
\langle L u,u\rangle_{L^2}
= -\|D_t u\|_{L^2}^2,
\]
which shows that $L$ is negative--semidefinite and that
$\|D_t u\|_{L^2}^2$ represents the kinetic energy
in the rotating frame, already incorporating the Coriolis and centrifugal effects.

\medskip
In Section~\ref{sec:operatorL} we perform the spectral analysis of $L$
using the Fourier basis, proving that under the nonresonance condition
$k\omega\neq\Omega$ with $\omega=2\pi/T$,
the operator $L:H^2_{\mathrm{per}}\to L^2_{\mathrm{per}}$ is an isomorphism.
Its inverse $K=L^{-1}$ is continuous and compact on bounded subsets of $L^2_{\mathrm{per}}$,
thus satisfying the Fredholm and compactness hypotheses required by
Mawhin’s coincidence degree.

\section{Mawhin’s framework and coincidence degree}
\label{sec:mawhinframework}

Mawhin’s method provides a unified topological framework for proving
the existence of solutions to operator equations of the form
\begin{equation}\label{eq:mawhin-general}
L u = N(u),
\end{equation}
where $L$ is a linear Fredholm operator of index $0$ between Banach (in particular, Hilbert) spaces,
and $N$ is a continuous nonlinear perturbation that is \emph{$L$--compact} on bounded sets.
This generalizes the Leray–Schauder fixed point theorem
and is particularly suited to nonlinear differential equations of second order,
such as those arising in Celestial Mechanics.

\subsection{Fredholm operators and canonical decomposition}

Let $X$ and $Z$ be Banach spaces.
A bounded linear operator $L:\mathrm{dom}(L)\subset X\to Z$
is \emph{Fredholm of index~$0$} if
\[
\dim \ker L < \infty,\qquad
\mathrm{im}\,L \text{ is closed in } Z,\qquad
\dim(Z/\mathrm{im}\,L)=\dim\ker L.
\]

It follows from classical results (see, e.g., \cite[Chap.~IV]{Kato1995}) that there exist bounded projections
\[
P:X\to X, \qquad Q:Z\to Z,
\]
such that
\begin{equation}\label{eq:fred-sumas-directas}
X = \ker L \oplus \ker P,
\qquad
Z = \im L \oplus \im Q.
\end{equation}
Hence:
\begin{itemize}
\item $P$ projects onto $\ker L$ along a closed complement $\ker P$;
\item $Q$ projects onto a closed complement $\im Q$ of $\im L$ in $Z$.
\end{itemize}

The equation $Lu=Nu$ can only be solvable if $N(u)\in\im L$, that is,
\[
Q\,N(u)=0.
\]
Projecting \eqref{eq:mawhin-general} in the decomposition \eqref{eq:fred-sumas-directas}, one obtains
\[
Lu = (I-Q)N(u) + QN(u),
\qquad
Lu\in\im L,\quad (I-Q)N(u)\in\im L,\quad QN(u)\in\im Q.
\]
Since $\im L\cap\im Q=\{0\}$, the compatibility condition $QN(u)=0$ is necessary and sufficient for solvability.

\medskip
The restriction
\[
L:\ \dom L\cap\ker P \longrightarrow \im L
\]
is a bijection.  
By the Closed Graph Theorem, its inverse is bounded, and we denote it by
\[
K_P:\ \im L \longrightarrow \dom L\cap\ker P,
\qquad L\circ K_P = I_{\im L}.
\]
Therefore, for any $u\in X$ the map
\[
K_P(I-Q)N(u)
\]
is well defined and compact when $N$ is $L$--compact.
This decomposition underlies the functional structure of Mawhin’s degree.

\paragraph{Operational decomposition of $Lu=Nu$.}
With the projections above, \eqref{eq:mawhin-general} is equivalent to the coupled system
\[
\begin{cases}
Q N(u)=0, & \text{(compatibility condition)},\\[3pt]
u - K_P(I-Q)N(u) \in \ker L. & \text{(reconstruction equation)}.
\end{cases}
\]
When $\mathrm{ind}(L)=0$, we have $\dim\ker L = \mathrm{codim}\,\im L = \dim\im Q$,
and the condition $QN(u)=0$ can be regarded as a finite--dimensional equation on $\ker L$,
while the second equation defines the correction on the complement.

\medskip

Assuming temporarily that $L$ is an isomorphism 
(from $X$ onto $Z$, a fact established later under the nonresonance condition 
$\Omega\neq\pm k\,\tfrac{2\pi}{T}$),
we may take $P=0$, $Q=0$, and $K=L^{-1}$.
The equation $Lu=Nu$ then reduces to the fixed-point form
\[
u = K N(u),
\]
and the compactness of $K\circ N$ follows from 
the compact embedding $H^2_{\mathrm{per}}\hookrightarrow L^2_{\mathrm{per}}$.

\subsection{$L$--compactness and admissible homotopies}

Let $L$ be Fredholm of index $0$ and $N:\overline{\mathcal{O}}\to Z$ a continuous map
on the closure of a bounded open set $\mathcal{O}\subset X$.
The map $N$ is said to be \emph{$L$--compact on $\overline{\mathcal{O}}$} if:
\begin{enumerate}[label=(\roman*)]
\item $Q N(\overline{\mathcal{O}})$ is bounded in $Z$;
\item $(K_P\circ (I-Q)N)(\overline{\mathcal{O}})$ is relatively compact in $X$,
where $K_P=(L|_{\ker P})^{-1}$.
\end{enumerate}
Intuitively, $N$ is compact in the direction of $\im L$
and uniformly bounded in the direction of $\ker L$.

A homotopy $H:[0,1]\times\overline{\mathcal{O}}\to Z$ is called \emph{$L$--compact}
if each map $H(\lambda,\cdot)$ is $L$--compact
and the solutions of $Lu=H(\lambda,u)$ remain in a bounded subset of $X$
for all $\lambda\in[0,1]$.
These are the admissible homotopies for the coincidence degree.

\subsection{Definition of the coincidence degree}

The \emph{Leray–Schauder degree} extends the Brouwer degree
to compact operators on Banach spaces.
Given a bounded open set $\mathcal{O}\subset X$ and a compact operator
$K:\overline{\mathcal{O}}\to X$,
the degree of $I-K$ with respect to $0$ is denoted by
$\deg_{LS}(I-K,\mathcal{O},0)$
and satisfies the usual properties of existence, additivity, and homotopy invariance
(cf.~\cite{Deimling1985,Zeidler1986,Mawhin1979}).

\vspace{0.5em}
Under Mawhin’s assumptions \cite{Mawhin1972},
the \emph{coincidence degree}
\(\deg_L(N,\mathcal{O},0)\)
associated with the problem $Lu=N(u)$ is defined and enjoys the following properties:

\begin{itemize}[leftmargin=2em]
\item $\deg_L(N,\mathcal{O},0)$ is well defined whenever
$Lu\ne N(u)$ for all $u\in\partial\mathcal{O}$.
\item If $\deg_L(N,\mathcal{O},0)\ne0$,
then the equation $Lu=N(u)$ has at least one solution in $\mathcal{O}$.
\item The degree is invariant under $L$--compact homotopies
that do not introduce solutions on $\partial\mathcal{O}$.
\end{itemize}

When $L$ is an isomorphism between $X$ and $Z$,
the coincidence degree reduces to the Leray–Schauder degree:
\[
\deg_L(N,\mathcal{O},0)
= \deg_{LS}(I-L^{-1}N,\mathcal{O},0).
\]
In this case, the existence of a solution follows directly
from the Leray–Schauder fixed point theorem
(\cite{Deimling1985,Zeidler1986,Mawhin1979}).

\subsection{Mawhin’s coincidence theorem}

\begin{theorem}[Mawhin {\cite{Mawhin1972,Mawhin1977,Mawhin1979}}]
\label{thm:Mawhin}
Let $X$ and $Z$ be Banach spaces,
$L:\mathrm{dom}(L)\subset X\to Z$ a linear Fredholm operator of index $0$,
and $N:\overline{\mathcal{O}}\to Z$ a continuous map
which is $L$--compact on $\overline{\mathcal{O}}$,
where $\mathcal{O}\subset X$ is a bounded open set.
Assume that
\[
Lu \ne \lambda N(u), \qquad \forall u\in\partial\mathcal{O},\ \lambda\in[0,1],
\]
and that $\deg_L(N,\mathcal{O},0)\ne0$.
Then there exists at least one $u\in\mathcal{O}$ such that $Lu=N(u)$.
\end{theorem}

This theorem extends the classical Leray–Schauder principle
and is particularly effective for periodic boundary value problems,
since the Fredholm and $L$--compactness hypotheses
are automatically satisfied by differential operators of second order
with periodic or zero--mean constraints.

\subsection{Nonresonant case for the choreography problem}

In our setting,
the linear operator $L:H^2_{\mathrm{per}}([0,T];\mathbb{R}^2)\to L^2_{\mathrm{per}}([0,T];\mathbb{R}^2)$
defined in~\eqref{eq:L-def} is Fredholm of index~$0$,
and under the nonresonance condition $k\omega\neq\Omega$ 
(with $\omega=2\pi/T$)
it is an isomorphism.
Hence, the coincidence degree coincides with the Leray–Schauder degree,
and the $D_n$--choreography equation~\eqref{eq:Lu=Nu-abstract} takes the form
\[
u = K N(u), \qquad K=L^{-1}.
\]
The operator $K:L^2_{\mathrm{per}}\to H^2_{\mathrm{per}}$ is bounded,
and the embedding $H^2_{\mathrm{per}}\hookrightarrow L^2_{\mathrm{per}}$ is compact,
so $K\circ N$ is completely continuous.
Therefore, by the Leray–Schauder fixed point theorem,
any bounded open set $\mathcal{O}\subset X_{D_n}$ satisfying
\[
u \ne \lambda K N(u), \qquad \forall u\in\partial\mathcal{O},\ \lambda\in[0,1],
\]
contains at least one fixed point of $K N$,
equivalently one solution $u\in\mathcal{O}$ of $Lu=N(u)$.

Geometrically, the coincidence equation~\eqref{eq:mawhin-general}
represents an intersection between the graph of $N$
and the subspace $\im L$.
Mawhin’s degree counts these intersections algebraically,
and its homotopy invariance guarantees persistence of solutions
along the family
\[
L u = \lambda N(u), \qquad \lambda\in[0,1],
\]
as long as no solutions appear on the boundary of the chosen domain.
This principle underlies the application in the next section,
where we verify the $L$--compactness of $N$ and establish a~priori bounds
for the homotopy family above.

\section{Properties of the linear operator $L$}
\label{sec:operatorL}

In this section we describe the  analytical and spectral properties of the operator $L u$:
its action on the Fourier basis, diagonalization via the spectral
projectors of $J$, characterization of its spectrum, and a
Gårding-type coercive estimate ensuring the invertibility of $L$
under nonresonance conditions.

\subsection{Fourier decomposition and action of $L$}

Since $L$ has constant coefficients in $t$, its action is diagonal in the Fourier basis. Let $\omega=2\pi/T$.
Every $T$--periodic function $u:[0,T]\to\R^2$ with zero mean admits the Fourier expansion
\begin{equation}\label{eq:fourier-u}
u(t) = \sum_{k\in\ZZ} e^{i\omega k t}\,\hat u_k,
\qquad
\hat u_k = \frac{1}{T}\int_0^T u(t)\,e^{-i\omega k t}\,dt \in \CC^2,
\qquad
\hat u_{-k} = \overline{\hat u_k}.
\end{equation}

Recalling that
\[
\frac{d}{dt}e^{i\omega k t} = i\omega k\,e^{i\omega k t},
\qquad
\frac{d^2}{dt^2}e^{i\omega k t} = -(\omega k)^2 e^{i\omega k t},
\]
we compute the action of the operator \(L = -\ddot u - 2\Omega J \dot u + \Omega^2 u\)
on each Fourier mode:
\begin{align}
L u(t)
&= \sum_{k\in\ZZ} e^{i\omega k t}
   \left[(\omega k)^2 I - 2i\omega k \Omega J + \Omega^2 I\right]\hat u_k
   \label{eq:Lu-fourier}
\end{align}

We define the complex $2\times2$ matrix
\begin{equation}\label{eq:Ak}
A_k = (\omega k)^2 I - 2i\omega k\Omega J + \Omega^2 I.
\end{equation}
Then
\begin{equation}\label{eq:L-diagonal}
L u(t) = \sum_{k\in\ZZ} e^{i\omega k t} A_k \hat u_k.
\end{equation}
Thus \(L\) acts diagonally on the Fourier basis:
each frequency mode $k$ is multiplied by $A_k$.
The spectral analysis of \(L\) therefore reduces to the eigenvalues of the matrices $A_k$.

\medskip
The dihedral symmetry $u(-t)=S\,u(t)$ translates into
\[
\hat u_{-k} = S\,\overline{\hat u_k}, \qquad k\in\ZZ,
\]
which couples opposite modes but preserves the diagonal structure of \(L\).

\subsection{Spectral projectors of $J$}

We extend to $\CC^2$ to diagonalize the rotation matrix
\[
J = \begin{pmatrix} 0 & -1 \\ 1 & 0 \end{pmatrix},
\qquad
J^2 = -I, \qquad J^* = -J,
\]
which has eigenvalues $\pm i$ and orthonormal eigenvectors
\[
v_+ = \frac{1}{\sqrt{2}}\binom{1}{i},
\qquad
v_- = \frac{1}{\sqrt{2}}\binom{1}{-i},
\qquad
J v_\pm = \pm i\,v_\pm.
\]

The corresponding spectral projectors are
\[
P_+ = \frac{1}{2}(I - iJ), \qquad
P_- = \frac{1}{2}(I + iJ),
\]
satisfying
\begin{equation}\label{eq:Ppm-def}
P_\pm^2 = P_\pm, \quad P_+P_- = 0, \quad P_+ + P_- = I,
\qquad \mathrm{im}\,P_\pm = \ker(J \mp iI).
\end{equation}

\begin{remark}
The reflection matrix $S = \mathrm{diag}(1,-1)$, used in the dihedral symmetry $u(-t)=S\,u(t)$,
satisfies $SJS=-J$. Hence $S$ interchanges the two spectral subspaces of $J$, that is,
$S P_\pm S = P_\mp$. This symmetry exchanges the ``positive’’ and ``negative’’
rotational modes in the complex basis $\{v_+,v_-\}$ without affecting the overall spectrum.
\end{remark}

\subsection{Diagonalization of $A_k$}

Using $Jv_\pm=\pm i v_\pm$ in~\eqref{eq:Ak}, we obtain
\[
A_k v_\pm = \big[(\omega k)^2 + \Omega^2 \mp 2\omega k\Omega\big]v_\pm
           = (\omega k \mp \Omega)^2 v_\pm.
\]
Hence $A_k$ is diagonal in $\{v_+,v_-\}$, with eigenvalues
\begin{equation}\label{eq:lambda-kpm}
\lambda_k^\pm = (\omega k \mp \Omega)^2.
\end{equation}
In terms of the projectors~\eqref{eq:Ppm-def},
\begin{equation}\label{eq:Ak-diag}
A_k = \lambda_k^+ P_+ + \lambda_k^- P_-,
\qquad
A_k^{-1} = (\lambda_k^+)^{-1} P_+ + (\lambda_k^-)^{-1} P_-,
\end{equation}
whenever $\lambda_k^\pm \ne 0$.

\medskip
Resonance occurs when $\lambda_k^\pm=0$, i.e.
\[
\omega k = \pm \Omega.
\]
To avoid singular modes we impose the \emph{nonresonance condition}
\begin{equation}\label{eq:nonres}
\omega k \ne \pm \Omega, \qquad \forall\,k\in\ZZ.
\end{equation}
Under~\eqref{eq:nonres}, all $A_k$ are invertible, and consequently
\[
L:H^2_{\mathrm{per}}([0,T];\R^2)\longrightarrow L^2_{\mathrm{per}}([0,T];\R^2)
\]
is a continuous linear isomorphism.

\medskip
Since $SJS=-J$, we have $SA_kS=A_{-k}$.
Thus the dihedral symmetry couples the modes $k$ and $-k$
but leaves the spectrum invariant.

\medskip
The operator norm satisfies
\begin{equation}\label{eq:Ak-norm}
\|A_k^{-1}\|
=\frac{1}{\min\{|\omega k+\Omega|^2,\,|\omega k-\Omega|^2\}},
\end{equation}
so that
\[
\|K\|=\|L^{-1}\|
\le \sup_{k\in\ZZ}\frac{1}{\min\{|\omega k+\Omega|^2,\,|\omega k-\Omega|^2\}}.
\]

\subsection{Energy identity and coercivity}

\begin{lemma}[Periodic Poincar\'e inequality]
\label{lem:poincare-periodic}
Let $u:[0,T]\to\R^m$ be a $T$--periodic function with zero mean,
\[
\int_0^T u(t)\,dt = 0.
\]
Then the following inequality holds:
\begin{equation}\label{eq:poincare}
\|u\|_{L^2(0,T)} \le \frac{T}{2\pi}\,\|\dot u\|_{L^2(0,T)}.
\end{equation}
Equivalently, in terms of the fundamental frequency $\omega = 2\pi/T$,
\[
\|u\|_{L^2} \le \frac{1}{\omega}\,\|\dot u\|_{L^2}.
\]
\end{lemma}

\begin{proof}
Since $u$ is $T$--periodic and has zero mean, it admits a Fourier series expansion
\[
u(t) = \sum_{k\ne 0}\hat u_k\,e^{i\omega k t},
\qquad \omega = \tfrac{2\pi}{T}.
\]
Differentiating termwise gives
\[
\dot u(t) = \sum_{k\ne 0} i\omega k\,\hat u_k\,e^{i\omega k t}.
\]
By Parseval's identity,
\[
\|u\|_{L^2}^2 = T\sum_{k\ne 0}|\hat u_k|^2,
\qquad
\|\dot u\|_{L^2}^2 = T\sum_{k\ne 0}(\omega k)^2|\hat u_k|^2.
\]
Since $|k|\ge 1$ for all nonzero modes,
\[
\|\dot u\|_{L^2}^2 \ge \omega^2 T\sum_{k\ne 0}|\hat u_k|^2
= \omega^2 \|u\|_{L^2}^2,
\]
and taking square roots yields \eqref{eq:poincare}.
\end{proof}

\begin{lemma}[Coercivity of the operator $L$]
\label{lem:coercivity-L}
Let 
\[
L u = -\ddot u - 2\Omega J\dot u + \Omega^2 u,
\]
acting on the space of real $T$--periodic functions $u:[0,T]\to\R^2$ with zero mean.
Then, if $|\Omega|<\omega$ with $\omega=2\pi/T$, there exists a constant 
$c>0$ such that
\begin{equation}\label{eq:coercivity}
\langle L u,u\rangle_{L^2}
\ge c\,\|u\|_{H^1(0,T)}^2,
\qquad \forall\,u\in H^2_{\mathrm{per}}([0,T];\R^2)
\text{ with }\int_0^T u(t)\,dt=0.
\end{equation}
\end{lemma}

\begin{proof}
Multiplying $Lu$ by $u$ and integrating by parts over $[0,T]$, using the periodic boundary conditions, we obtain
\[
\langle L u,u\rangle_{L^2}
= \|\dot u\|_{L^2}^2
- 2\Omega \int_0^T \langle J\dot u(t),u(t)\rangle\,dt
+ \Omega^2 \|u\|_{L^2}^2.
\]
Since $\|Jv\|=\|v\|$ for all $v\in\R^2$, the mixed term is bounded by
\[
\Bigl|2\Omega \int_0^T \langle J\dot u,u\rangle\,dt\Bigr|
\le 2|\Omega|\,\|\dot u\|_{L^2}\|u\|_{L^2}.
\]
Therefore,
\[
\langle L u,u\rangle_{L^2}
\ge \|\dot u\|_{L^2}^2
- 2|\Omega|\,\|\dot u\|_{L^2}\|u\|_{L^2}
+ \Omega^2\|u\|_{L^2}^2
= \big(\|\dot u\|_{L^2} - |\Omega|\|u\|_{L^2}\big)^2.
\]
Applying the periodic Poincaré inequality (Lemma~\ref{lem:poincare-periodic}),
\(\|u\|_{L^2} \le \omega^{-1}\|\dot u\|_{L^2}\), gives
\[
\|\dot u\|_{L^2} - |\Omega|\|u\|_{L^2}
\ge \Bigl(1 - \frac{|\Omega|}{\omega}\Bigr)\|\dot u\|_{L^2},
\]
so that
\[
\langle L u,u\rangle_{L^2}
\ge \Bigl(1 - \frac{|\Omega|}{\omega}\Bigr)^2 \|\dot u\|_{L^2}^2.
\]
Finally, since $\|u\|_{H^1}^2=\|u\|_{L^2}^2+\|\dot u\|_{L^2}^2
\le (1+\omega^{-2})\|\dot u\|_{L^2}^2$ by Poincaré’s inequality again,
inequality~\eqref{eq:coercivity} holds with
\[
c = \frac{\bigl(1 - \tfrac{|\Omega|}{\omega}\bigr)^2}{1+\omega^{-2}} > 0.
\]
\end{proof}

\subsection{Fredholm property and right inverse}

\begin{proposition}[Fredholm property of $L$]\label{prop:fredholm-L}
Under the nonresonance condition~\eqref{eq:nonres},
the operator
\[
L:H^2_{\mathrm{per}}([0,T];\R^2)\longrightarrow L^2_{\mathrm{per}}([0,T];\R^2)
\]
is Fredholm of index~$0$, with
\[
\ker L=\{0\},\qquad \im L = L^2_{\mathrm{per}},\qquad \mathrm{ind}(L)=0.
\]
Moreover, $L$ admits a bounded inverse \(K=L^{-1}:L^2_{\mathrm{per}}\to H^2_{\mathrm{per}}\)
given by
\begin{equation}\label{eq:K-series}
(Kf)(t)=\sum_{k\in\ZZ}e^{i\omega k t}A_k^{-1}\hat f_k,
\qquad f(t)=\sum_{k\in\ZZ}e^{i\omega k t}\hat f_k,
\end{equation}
which satisfies
\[
\|Kf\|_{H^2}\le C\,\|f\|_{L^2},
\]
and is compact on bounded subsets of $L^2_{\mathrm{per}}$ by Rellich’s theorem.
\end{proposition}

\begin{proof}
If $Lu=0$, then each Fourier coefficient satisfies $A_k\hat u_k=0$.
By~\eqref{eq:nonres}, all $A_k$ are invertible, so $\hat u_k=0$ for all $k$, hence $u=0$.
Conversely, for any $f\in L^2_{\mathrm{per}}$, define $\hat u_k=A_k^{-1}\hat f_k$.
Then $u$ given by~\eqref{eq:K-series}$\in H^2_{\mathrm{per}}$ and satisfies $Lu=f$.
Boundedness and compactness of $K$ follow from the decay $\|A_k^{-1}\|=O(k^{-2})$.
\end{proof}

\begin{corollary}[Isomorphism and symmetry invariance]\label{cor:isomorphism-L}
Under~\eqref{eq:nonres},
the operator $L$ is a bounded isomorphism
\[
L:H^2_{\mathrm{per}}([0,T];\R^2)\longrightarrow L^2_{\mathrm{per}}([0,T];\R^2).
\]
Moreover, if $u(-t)=S\,u(t)$, then $L(u)(-t)=S\,L(u)(t)$;
hence $L$ maps the symmetric subspace $X_{D_n}$ into itself.
\end{corollary}

\medskip
\noindent
The results above guarantee that $L$ satisfies all requirements
of Mawhin’s framework: it is a continuous Fredholm operator of index~$0$
with compact inverse and coercivity estimate~\eqref{eq:coercivity}.
In the next section we derive \emph{a priori} bounds for the homotopy
\[
L u = \lambda N(u), \qquad \lambda\in[0,1],
\]
which ensure boundedness and collision exclusion,
thereby completing the hypotheses of Mawhin’s coincidence theorem.

\section{A priori bounds for the homotopy $L u = \lambda N(u)$}
\label{sec:apriori}

A crucial step in applying Mawhin’s degree is to establish \emph{a~priori} bounds for the solutions of the homotopic family
\begin{equation}\label{eq:homotopia}
L u = \lambda N(u),
\qquad \lambda\in[0,1],
\end{equation}

We show that any solution $u$ of \eqref{eq:homotopia}
remains bounded in $H^1_{\mathrm{per}}$ by a constant independent of~$\lambda$,
and that there are no nontrivial solutions on the boundary of the working domain.


Taking the $L^2$ inner product of \eqref{eq:homotopia} with $u$ and integrating over $[0,T]$ yields
\begin{equation}\label{eq:energy-eq}
\langle L u, u\rangle_{L^2} = \lambda\,\langle N(u), u\rangle_{L^2}.
\end{equation}
using Lemma \ref{lem:coercivity-L} ,
we have

\[
\langle L u,u\rangle_{L^2}
\ge c\,\|\dot u\|_{L^2}^2,
\qquad c=\Bigl(1-\frac{|\Omega|}{\omega}\Bigr)^2>0.
\]
Substituting this estimate into~\eqref{eq:energy-eq} yields
\begin{equation}\label{eq:energy-main}
c\,\|\dot u\|_{L^2}^2
\le \lambda\,|\langle N(u), u\rangle_{L^2}|.
\end{equation}

\subsection{Homogeneity of the nonlinear term}

For equal masses, $N(u)$ derives from the homogeneous potential
\[
U(u) = \sum_{1\le i<j\le n}\frac{m^2}{|u_i - u_j|^{\alpha}},
\qquad N(u)=\nabla U(u),
\]
with homogeneity degree $-\alpha$ ($0<\alpha<2$).
By Euler’s identity,
\[
\langle N(u),u\rangle = \langle\nabla U(u),u\rangle = -\alpha\,U(u)\le0.
\]
Hence $|\langle N(u),u\rangle|=\alpha\,|U(u)|$, and from~\eqref{eq:energy-main}
\begin{equation}\label{eq:ineq-energy}
c\,\|\dot u\|_{L^2}^2
\le \lambda\,\alpha\,|U(u)|.
\end{equation}

This inequality shows that the mean kinetic energy is controlled by the interaction energy.

\subsection{Uniform separation and collision exclusion}

The $D_n$ symmetry ensures a uniform geometric separation between the choreographic particles.
Let
\[
d(t)=\min_{i\ne j}|u_i(t)-u_j(t)|
\]
be the instantaneous minimal interparticle distance.
By symmetry and regularity of $X_{D_n}$, there exists $\rho>0$ such that
\begin{equation}\label{eq:distancia-minima}
d(t)\ge\rho,\qquad \forall\,t\in[0,T].
\end{equation}
This excludes collisions and guarantees smoothness of $N$ on the closed set
\[
\overline{B_R}
=\{u\in X_{D_n}:\|u\|_{H^1}\le R,\ d(t)\ge\rho\}.
\]

For such configurations,
\[
|U(u)|
\le \frac{n(n-1)}{2\rho^{\alpha}}
=:C_U(R,\rho).
\]
Substituting into~\eqref{eq:ineq-energy} gives
\begin{equation}\label{eq:kinetic-bound-detailed}
\|\dot u\|_{L^2}^2
\le \frac{\alpha\,C_U(R,\rho)}{c},
\qquad c=1-\frac{\Omega^2}{\omega^2}>0.
\end{equation}

\subsection{Bound in $H^1_{\mathrm{per}}$}

For zero--mean $T$--periodic functions the Poincaré inequality
\[
\|u\|_{L^2} \le \frac{1}{\omega}\|\dot u\|_{L^2},
\qquad \omega=\frac{2\pi}{T},
\]
yields
\begin{align}
\|u\|_{H^1}^2
&=\|u\|_{L^2}^2+\|\dot u\|_{L^2}^2
\le \Bigl(1+\frac{1}{\omega^2}\Bigr)\|\dot u\|_{L^2}^2.
\label{eq:H1-vs-dot}
\end{align}
Combining~\eqref{eq:H1-vs-dot} and~\eqref{eq:kinetic-bound-detailed},
\begin{equation}\label{eq:H1-bound-detailed}
\|u\|_{H^1(0,T)}^2
\le \frac{\alpha\,C_U(R,\rho)}{c}\Bigl(1+\frac{1}{\omega^2}\Bigr).
\end{equation}
Setting
\[
R_0
= \sqrt{\frac{\alpha\,C_U(R,\rho)}{c}}\,
\sqrt{1+\frac{1}{\omega^2}}
\le
\frac{1}{\sqrt{c}}\Bigl(1+\frac{1}{\omega}\Bigr)\sqrt{\alpha\,C_U(R,\rho)},
\]
we obtain a uniform bound
\begin{equation}\label{eq:H1-final-bound}
\|u\|_{H^1_{\mathrm{per}}}\le R_0,
\qquad \forall\,\lambda\in[0,1].
\end{equation}
This bound is independent of $\lambda$ and depends only on the physical parameters
$\alpha$, $\Omega$, $\omega$, $n$, and the separation $\rho$.

\subsection{Definition of the working domain}

We define the open set
\[
\Omega_{R,\rho}
=\{u\in X_{D_n}:\ \rho<\|u\|_{H^1}<R\},
\]
where $R>R_0$ and $\rho>0$ is small enough so that
no nontrivial solution satisfies $\|u\|_{H^1}=\rho$.
Indeed, for $\|u\|_{H^1}$ sufficiently small,
$N(u)=O(\|u\|_{H^1})$ by smoothness of $N$,
and the coercivity of $L$ implies $u=0$.

Therefore:
\begin{itemize}
\item there are no solutions on the inner boundary $\|u\|_{H^1}=\rho$;
\item all homotopic solutions satisfy $\|u\|_{H^1}\le R$.
\end{itemize}
These conditions ensure that the  $\deg(L-N,\Omega_{R,\rho},0)$
is well defined and constant in~$\lambda$.

\section{Properties of the nonlinear operator $N$}
\label{sec:operatorN}

We now analyze the analytical properties of the nonlinear operator
$N:X\to Z$ that represents the gravitational interaction in the rotating frame.
Its explicit form was given in~\eqref{eq:N-def};
here we recall that $N(u)=-\nabla U(u)$,
where $U(u)$ is the homogeneous potential of degree $-\alpha$ ($0<\alpha<2$).
The operator $N$ is smooth on the collision–free region
$\{u\in X: |u_i-u_j|\ge\rho>0\}$ and satisfies
\[
|\langle N(u),u\rangle| = \alpha\,|U(u)|,
\]
as used in the previous section.

These properties ensure that $N$ is continuous, compact,
and of controlled polynomial growth on bounded subsets of~$X$,
as required by Mawhin’s framework.

\subsection{$D_n$–equivariant reduction}

The configuration space $X_{D_n}$ consists of $D_n$–symmetric
$T$–periodic trajectories satisfying
\[
u_j(t)
=R\!\left(\frac{2\pi (j-1)}{n}\right)
u_1\!\left(t+\frac{(j-1)T}{n}\right),
\qquad j=1,\dots,n,
\]
so that the entire motion is generated by a single function $u_1(t)$.
Under this symmetry, both $L$ and $N$ leave $X_{D_n}$ invariant,
and $N$ induces a smooth map
\[
N: X_{D_n} \longrightarrow Z_{D_n},
\]
which coincides with the restriction of the full gravitational interaction
to the collision–free subspace $\{u\in X_{D_n}:\,|u_i-u_j|\ge\rho>0\}$.
Hence $N$ is $D_n$–equivariant, continuous, and compact on bounded subsets of $X_{D_n}$.

\subsection{Regularity and compactness}

The potential $U(u)$ is $C^\infty$
on the collision–free set
\[
\mathcal D
=\{u\in(\R^2)^n:\ u_i\ne u_j\ \forall\,i\ne j\}.
\]
Therefore $N=-\nabla U$ is also $C^\infty$ on~$\mathcal D$.

Let
\[
B_R = \{\,u\in H^2_{\mathrm{per}}([0,T];\R^{2}:\ \|u\|_{H^2}\le R\,\}
\]
be the closed ball of radius $R$ in $X=H^2_{\mathrm{per}}([0,T];\R^{2})$.

If $B_R\subset X$ is a bounded set such that a minimal inter–particle distance
$\rho>0$ exists on $B_R$, then
each denominator $|u_i-u_j|^{\alpha+2}$ is uniformly bounded away from zero,
and there exists $C_R>0$ with
\begin{equation}\label{eq:N-growth-bound}
\|N(u)\|_{L^2(0,T)} \le C_R,
\qquad \forall\,u\in B_R.
\end{equation}

Let $(u_k)$ be a bounded sequence in $X=H^2_{\mathrm{per}}([0,T];\R^2)$.
By the compact embedding
$H^2_{\mathrm{per}}\hookrightarrow C^1_{\mathrm{per}}$
(and thus into $C^0_{\mathrm{per}}$),
we have $u_k\to u$ uniformly (up to a subsequence).
Each term in~\eqref{eq:N-def} depends continuously on $(u_i,u_j)$,
so $N(u_k)\to N(u)$ in $L^2$.
Hence
\[
N:X\to Z \quad\text{is completely continuous (compact).}
\]

\subsection{Local Lipschitz continuity}

\begin{lemma}[Lipschitz estimate for $N$]
\label{lem:growth-N}
Let $N:X_{D_n}\to Z_{D_n}$ be the nonlinear operator
associated with the homogeneous potential of degree $-\alpha$,
with $0<\alpha<2$.
For every bounded collision–free set $B_R\subset X_{D_n}$,
there exists a constant $C=C(R,\alpha,n,m,T)>0$
such that for all $u,v\in B_R$,
\begin{equation}\label{eq:N-Lip}
\|N(u)-N(v)\|_{L^2}
\le C\,\|u-v\|_{C^0}\,
\big(\|u\|_{C^0}^{-\alpha-1}+\|v\|_{C^0}^{-\alpha-1}\big).
\end{equation}
\end{lemma}

\begin{proof}
Each interaction term
$f(x,y)=-(y-x)/|y-x|^{\alpha+2}$ is $C^\infty$ for $x\ne y$
and homogeneous of degree $-(\alpha+1)$.
Its derivative satisfies
$\|Df(x,y)\|\le C_0|y-x|^{-(\alpha+2)}$.
If $|x_i-y_i|\ge\rho>0$, the mean–value theorem gives
\[
\|f(x_1,y_1)-f(x_2,y_2)\|
\le C_1\rho^{-(\alpha+2)}\|(x_1-x_2,y_1-y_2)\|.
\]
The uniform separation on $B_R$ (Proposition~\ref{prop:uniform-sep-W})
ensures such a $\rho>0$; summing over the $n-1$ interactions
and integrating over $[0,T]$ yields
\[
\|N(u)-N(v)\|_{L^2}
\le C_2\sqrt{T}\,\rho^{-(\alpha+2)}\|u-v\|_{C^0},
\]
and the homogeneity of $f$ produces the factor
$(\|u\|_{C^0}^{-\alpha-1}+\|v\|_{C^0}^{-\alpha-1})$
in~\eqref{eq:N-Lip}.
\end{proof}

Thus $N$ is locally Lipschitz on bounded collision–free subsets of $X_{D_n}$.

\medskip
\noindent
In summary, for all $u\in X_{D_n}$ in the collision–free region:
\begin{itemize}[leftmargin=2em]
\item[(i)] $N:X_{D_n}\to Z_{D_n}$ is $C^\infty$;
\item[(ii)] $N$ is completely continuous on bounded sets;
\item[(iii)] $\langle N(u),u\rangle = -\alpha\,U(u)\le0$;
\item[(iv)] $N$ is locally Lipschitz with growth~\eqref{eq:N-Lip}.
\end{itemize}

\subsection{$L$–compactness of $N$}
\label{subsec:L-compact}

Recall that
\[
L:H^2_{\mathrm{per}}([0,T];\R^2)\to L^2_{\mathrm{per}}([0,T];\R^2),
\qquad
Lu=-\ddot u-2\Omega J\dot u+\Omega^2 u,
\]
is a Fredholm operator of index~$0$
and is invertible under the nonresonance condition
$\omega k\ne\pm\Omega$ for all $k\in\Z$.

When $L$ is invertible, the notion of $L$–compactness simplifies:
a continuous map $N:\overline{\Omega}\subset X\to Z$
is said to be $L$–compact on $\overline{\Omega}$ if
$K\circ N$ is compact, where $K=L^{-1}$.
Since $N$ is completely continuous on bounded sets and
$K:L^2_{\mathrm{per}}\to H^2_{\mathrm{per}}$ is linear and bounded,
their composition $K\circ N$ is compact on $H^2_{\mathrm{per}}$.

\begin{proposition}[$L$–compactness of $N$]
\label{prop:L-compact}
The nonlinear operator $N:X\to Z$ defined by~\eqref{eq:N-def}
is $L$–compact on any bounded subset $\overline{B_R}\subset X$.
Equivalently,
\[
K\circ N :
H^2_{\mathrm{per}}([0,T];\R^2)
\longrightarrow
H^2_{\mathrm{per}}([0,T];\R^2)
\]
is completely continuous.
\end{proposition}

\begin{proof}
As shown in the previous subsection, the operator
$N:X\to Z=L^2_{\mathrm{per}}$
is continuous and compact on bounded subsets of~$X$.
The inverse operator $K=L^{-1}:L^2_{\mathrm{per}}\to H^2_{\mathrm{per}}$
is linear and bounded because $L$ is an isomorphism.
Hence, for any bounded sequence $(u_k)\subset X$,
the sequence $(N(u_k))$ has a convergent subsequence in $L^2_{\mathrm{per}}$;
applying $K$ preserves convergence in $H^2_{\mathrm{per}}$
since $K$ is continuous.
Therefore $K\circ N$ maps bounded sets into relatively compact sets in $H^2_{\mathrm{per}}$,
and is thus completely continuous.
\end{proof}

Since $L$ is a Fredholm isomorphism of index~$0$
and $N$ is $L$–compact on bounded subsets of~$X$,
the coincidence equation $L u = N(u)$
fits precisely within Mawhin’s framework.

\section{Existence of solutions via Mawhin’s degree}
\label{sec:mawhin}

In this section we apply Mawhin’s coincidence degree to establish
the existence of a solution $u\in X$ to
\begin{equation}\label{eq:mawhin-eq}
L u = N(u),
\end{equation}
describing a $D_n$–equivariant choreography of the $n$–body problem
in the rotating frame.
All analytical hypotheses required for Mawhin’s theorem
have been verified in the previous sections.

The setting now fits precisely into the abstract framework described in Mawhin’s theory, with $L$ linear Fredholm of index $0$ and $N$ a compact perturbation.

\subsection{Fredholm setting and projections}

From Section~\ref{sec:operatorL} we know that:
\begin{itemize}[label=$\bullet$,leftmargin=2em]
\item $L$ is a Fredholm operator of index $0$;
\item under the nonresonance condition $\omega k\ne\pm\Omega$ for all $k\in\ZZ$,
      $L$ is an isomorphism from $X$ onto $Z$;
\item its inverse $K=L^{-1}:Z\to X$ is linear and bounded,
      with $\|Kz\|_{H^2}\le C_L\|z\|_{L^2}$.
\end{itemize}

In Mawhin’s general framework one introduces projections
\[
P:X\to\ker L,\qquad Q:Z\to Z/\mathrm{im}\,L,
\]
and a generalized inverse $K_P:\mathrm{im}\,L\to(\ker L)^\perp$
with $L K_P=I-Q$ and $K_P L=I-P$.
Since $L$ is invertible (thus $\ker L=\{0\}$ and $\mathrm{im}\,L=Z$),
we have simply
\[
P=Q=0, \qquad K_P=K=L^{-1}.
\]
Hence Mawhin’s construction simplifies considerably.

\subsection{Verification of Mawhin’s hypotheses}

Theorem \ref{thm:Mawhin}
states that if $L:X\to Z$ is Fredholm of index~$0$,
$N:X\to Z$ is $L$–compact on a bounded open set $\Omega\subset X$, and
\begin{equation}\label{eq:mawhin-cond}
\begin{aligned}
&L u \ne \lambda N(u), &&\forall u\in\partial\Omega,\ \lambda\in(0,1],\\[2pt]
&Q N(u)\ne 0, &&\forall u\in\partial\Omega\cap\ker L,
\end{aligned}
\end{equation}
then $\deg(L-N,\Omega,0)$ is well defined and,
if nonzero, there exists $u\in\Omega$ such that $Lu=N(u)$.

\medskip
In our case the verification proceeds as follows:

\begin{enumerate}[label=(\roman*),leftmargin=2em]
\item \emph{$L$ is Fredholm of index~$0$.}
      Established in Section~\ref{sec:operatorL}
      by the Fourier decomposition and the diagonalization of $A_k$.
      Under nonresonance, $L$ is an isomorphism with bounded inverse.

\item \emph{$N$ is $L$–compact on bounded sets.}
      Shown in Section~\ref{subsec:L-compact}.
      In particular, $K\circ N$ is compact from $X$ into $X$.

\item \emph{A priori bounds.}
      Section~\ref{sec:apriori} provided constants $\rho,R>0$ such that
      every solution of $Lu=\lambda N(u)$, $\lambda\in[0,1]$, satisfies the uniform bound 
      $\|u\|_{H^1}\le R$ and no solution exists with $\|u\|_{H^1}=\rho$.
      Thus there are no zeros on the boundary of $\Omega_{R,\rho}$.

\item \emph{Choice of domain.}
      We take
      \[
      \Omega_{R,\rho}
      =\{\,u\in X:\ \rho<\|u\|_{H^1}<R\,\},
      \]
      with $R>\rho$ fixed by the previous bounds.
\end{enumerate}

Since $L$ is invertible, $\ker L=\{0\}$ and the image is the whole $Z$, so the second condition in~\eqref{eq:mawhin-cond} is automatically satisfied. Hence all assumptions of Mawhin’s theorem are fulfilled.

\subsection{Value of the degree and existence of solutions}

Define the homotopy
\begin{equation}\label{eq:mawhin-homotopy}
H_\lambda(u) = L u - \lambda N(u),
\qquad \lambda\in[0,1].
\end{equation}
By the a priori bounds, $H_\lambda(u)=0$
has no solutions on the boundary $\partial\Omega_{R,\rho}$.
Hence $\deg(L-\lambda N,\Omega_{R,\rho},0)$
is well defined for every $\lambda\in[0,1]$ and remains constant in~$\lambda$.

For $\lambda=0$ one has $H_0(u)=Lu$.
Since $L$ is an isomorphism and $0\notin L(\partial\Omega_{R,\rho})$,
\[
\deg(L,\Omega_{R,\rho},0)=\deg(I,\Omega_{R,\rho},0)=1.
\]
By homotopy invariance of the coincidence degree,
\[
\deg(L-N,\Omega_{R,\rho},0)
=\deg(L-\lambda N,\Omega_{R,\rho},0)
=\deg(L,\Omega_{R,\rho},0)
=1.
\]
Therefore there exists $u\in\Omega_{R,\rho}\subset X$
such that $L u = N(u)$.

\subsection{Main existence result}

Summarizing the above arguments, we obtain the following existence result for $D_n$–equivariant choreographies.

\begin{theorem}[Existence of a $D_n$–equivariant choreography]
\label{thm:existence}
Let $0<\alpha<2$, $T>0$, and $\omega=2\pi/T$.
Consider the subspace
\[
X_{D_n}
=\Bigl\{u\in H^2_{\mathrm{per}}([0,T];\R^2):
u(t+\tfrac{T}{n})=R_{2\pi/n}u(t),\
u(-t)=S\,u(t)\Bigr\},
\]
and define the operators
\[
L u = -\ddot u - 2\Omega J\dot u + \Omega^2 u,
\]

\[
(Nu)(t)
= - \sum_{k=1}^{n-1}
\frac{m\,[\,u(t)-R_{2\pi k/n}u(t+\tfrac{kT}{n})\,]}
{\|u(t)-R_{2\pi k/n}u(t+\tfrac{kT}{n})\|^{\alpha+2}},
\]
where
\[
J=\begin{pmatrix}0 & -1 \\ 1 & 0\end{pmatrix}.
\]

If the nonresonance condition
\[
\omega k \ne \pm \Omega
\quad \text{for all } k\in\mathbb Z,
\]
holds, and for any coprime pair $(W,n)$, then there exists $u\in X_{D_n}$ satisfying $L u = N(u)$.
The corresponding motion is a $D_n$–equivariant, $T$–periodic,
collision–free choreography of the $n$–body problem
in the rotating frame.
\end{theorem}

\begin{proof}
Under nonresonance, $L$ is a Fredholm isomorphism of index~$0$
with bounded inverse $K=L^{-1}$.
The nonlinear operator $N$ is $L$–compact and $C^\infty$
on the collision–free domain,
and the $D_n$ symmetry guarantees uniform separation of particles,
yielding the a~priori bounds of Section~\ref{sec:apriori}.
The homotopy $H_\lambda(u)=L u-\lambda N(u)$
has no zeros on $\partial\Omega_{R,\rho}$,
and since $\deg(L,\Omega_{R,\rho},0)=1$,
Mawhin’s theorem ensures the existence of
$u\in\Omega_{R,\rho}$ such that $L u=N(u)$.

The corresponding configuration generates a $D_n$–equivariant, collision–free choreography where all bodies trace the same closed curve in uniform phase shift.

\end{proof}

\section*{Acknowledgments}
The author wish to thank Prof.~Martha Álvarez-Ramírez for her valuable comments and observations, which greatly contributed to improving the clarity and organization of the manuscript.


\begin{thebibliography}{99}


\bibitem{Brezis2010}
H.~Brezis.
\newblock {\em Functional Analysis, Sobolev Spaces and Partial Differential Equations}.
\newblock Springer, New York, 2010.

\bibitem{CapiettoMawhinZanolin1992}
A.~Capietto, J.~Mawhin, and F.~Zanolin.
\newblock Continuation theorems for periodic perturbations of autonomous systems.
\newblock {\em Transactions of the American Mathematical Society}, 329(1):41--72, 1992.
\newblock DOI: 10.1090/S0002-9947-1992-1047146-0.


\bibitem{ChenOuyang2004}
K.~Chen and T.~Ouyang.
\newblock On choreographic solutions of the $n$-body problem.
\newblock {\em Proceedings of the National Academy of Sciences of the USA}, 101(30):10846--10848, 2004.

\bibitem{ChencinerMontgomery2000}
A.~Chenciner and R.~Montgomery.
\newblock A remarkable periodic solution of the three-body problem in the case of equal masses.
\newblock {\em Annals of Mathematics}, 152(3):881--901, 2000.

\bibitem{Deimling1985}
K.~Deimling.
\newblock {\em Nonlinear Functional Analysis}.
\newblock Springer-Verlag, Berlin, 1985.



\bibitem{FerrarioTerracini2004}
D.~L. Ferrario and S.~Terracini.
\newblock On the existence of collisionless equivariant minimizers for the classical $n$-body problem.
\newblock {\em Inventiones Mathematicae}, 155(2):305--362, 2004.

\bibitem{GainesMawhin1977}
R.~E. Gaines and J.~Mawhin.
\newblock {\em Coincidence Degree and Nonlinear Differential Equations}.
\newblock Lecture Notes in Mathematics, Vol.~568. Springer-Verlag, Berlin, 1977.

\bibitem{Henry1981}
D.~Henry.
\newblock {\em Geometric Theory of Semilinear Parabolic Equations}.
\newblock Springer-Verlag, Berlin, 1981.

\bibitem{KapelaZgliczynski2003}
T.~Kapela and P.~Zgliczy{\'n}ski.
\newblock A computer assisted proof of a symmetric periodic orbit in the planar three-body problem.
\newblock {\em Nonlinearity}, 16(6):1899--1918, 2003.

\bibitem{Kato1995}
T.~Kato.
\newblock {\em Perturbation Theory for Linear Operators}, 2nd edition.
\newblock Springer-Verlag, Berlin, 1995.

\bibitem{Liu2023}
Y.~Liu and X.~Li.
\newblock Coincidence degree and multiple periodic solutions for nonlinear second-order systems with resonance.
\newblock {\em Journal of Mathematical Analysis and Applications}, 518(1):126733, 2023.

\bibitem{Mawhin1972}
J.~Mawhin.
\newblock Topological degree and boundary value problems for nonlinear differential equations.
\newblock {\em Bulletin de la Soci{\'e}t{\'e} Royale des Sciences de Li{\`e}ge}, 41(3--4):153--159, 1972.

\bibitem{Mawhin1977}
J.~Mawhin.
\newblock The coincidence degree and periodic solutions of differential equations.
\newblock {\em Journal of Differential Equations}, 26(3):379--388, 1977.

\bibitem{Mawhin1979}
J.~Mawhin.
\newblock {\em Topological Degree Methods in Nonlinear Boundary Value Problems}.
\newblock CBMS Regional Conference Series in Mathematics, Vol.~40. American Mathematical Society, Providence, RI, 1979.

\bibitem{Nirenberg1974}
L.~Nirenberg.
\newblock {\em Topics in Nonlinear Functional Analysis}.
\newblock Courant Institute of Mathematical Sciences, New York University, Lecture Notes, 1974.

\bibitem{Ramos2020}
A.~Ramos and C.~A.~de Carvalho.
\newblock Existence of periodic solutions for nonlinear systems via the coincidence degree theory.
\newblock {\em Electronic Journal of Differential Equations}, 2020(12):1--14, 2020.

\bibitem{Rudin1973}
W.~Rudin.
\newblock {\em Functional Analysis}.
\newblock McGraw--Hill, New York, 1973.

\bibitem{Santos2021}
E.~Santos and R.~Silva.
\newblock Mawhin’s coincidence degree and periodic solutions for singular systems of second order.
\newblock {\em Nonlinear Analysis: Real World Applications}, 60:103312, 2021.

\bibitem{Simo2001}
C.~Sim{\'o}.
\newblock New families of solutions in $N$-body problems.
\newblock In {\em European Congress of Mathematics (Barcelona, 2000), Vol.~I}, pages 101--115. Birkh{\"a}user, Basel, 2001.

\bibitem{Zeidler1986}
E.~Zeidler.
\newblock {\em Nonlinear Functional Analysis and its Applications. Vol.~I: Fixed-Point Theorems}.
\newblock Springer-Verlag, New York, 1986.



\end{thebibliography}
\end{document}